\documentclass[12pt,a4paper,leqno]{amsart}

\usepackage{amsmath}
\usepackage{amssymb,hyperref}
\usepackage{multicol}
\usepackage{array,float}
\usepackage{amsthm,amsbsy,amscd,mathrsfs}
\usepackage{graphicx}
\usepackage{mathtools}
\usepackage{cite}
\usepackage{caption}
\usepackage{xcolor}
 \usepackage{tikz} 
 \usepackage{lscape}

\makeatletter
\newcommand\suchthat{%
 \@ifstar
  {\mathrel{}\middle|\mathrel{}}
  {\mid}%
}
\makeatother

\def\bf{\mathbf}

\def\rm{\textrm}

\def\lcm{\mathrm{lcm}} 

\newtheorem{theorem}{Theorem}[section]
  \newtheorem{proposition}[theorem]{Proposition}
  \newtheorem{corollary}[theorem]{Corollary}
  \newtheorem{lemma}[theorem]{Lemma}
  
  \newtheorem{definition}{Definition}[section]
 \theoremstyle{remark}
\newtheorem{remark}[theorem]{Remark}
\numberwithin{equation}{section}

\DeclareMathOperator{\PicoF}{Pic^0_F}
  
  \DeclareMathOperator{\DivoF}{Div^0_F}
  
  \DeclareMathOperator{\To}{T^0}
  \DeclareMathOperator{\T}{T}
  \DeclareMathOperator{\Tr}{Tr}  
\DeclareMathOperator{\cl}{Cl}

 \usepackage{mathtools}

 \title{THE SIZE FUNCTION FOR IMAGINARY CYCLIC SEXTIC FIELDS}
 
 \author{Ha Thanh Nguyen  Tran}
\address{Department of Mathematical and Physical Sciences\\
	Concordia University of Edmonton, 	7128 Ada Blvd NW\\
	 Edmonton,	AB T5B 4E4, Canada}
\email{hatran1104@gmail.com}

\author{Peng Tian} 
\address{Department of Mathematics\\
	East China University of Science and Technology\\
	Meilong Road 130, 200237, Shanghai, P. R. China}
\email{tianpeng@ecust.edu.cn}

\author{Amy Feaver} 
\address{Department of Mathematics and Computer Science\\
	Gordon College\\
	255 Grapevine Rd \\
	Wenham, MA 01984}
\email{mathfeaver@gmail.com}



\begin{document}

\begin{abstract}
 In this paper, we investigate the size function $h^0$ for number fields. This size function is analogous to the dimension of the Riemann-Roch spaces of divisors on an algebraic curve. Van der Geer and Schoof conjectured that $h^0$ attains its maximum at the trivial class of Arakelov divisors. This conjecture was proved for all number fields with unit group of rank 0 and 1, and also for cyclic cubic fields which have unit group of rank two. We prove the conjecture also holds for totally imaginary cyclic sextic fields, another class of number fields with unit group of rank two.
\end{abstract}

\maketitle

\section{Introduction}\label{sec1a}

The size function $h^0$ for a number field $F$ is well-defined on the Arakelov class group $\PicoF$ of $F$ \cite{ref:4}. This function was first introduced by van der Geer and Schoof \cite{ref:3} and also by Groenwegen  \cite{groenewegen-thesis, groenewegen-paper}. Van der Geer and Schoof conjectured that $h^0$ assumes its maximum on the trivial class $O_F$, the ring of integers of $F$, whenever $F/\mathbb{Q}$ is Galois or $F$ is Galois over an imaginary quadratic field \cite{ref:3}. Experiments showed  that this conjecture is true \cite{Tran4}.

By 2004 Francini proved the conjecture for all imaginary and   real quadratic fields \cite{ref:14} and certain pure cubic fields \cite{ref:15}. This establishes the conjecture for fields with unit groups of rank zero and some with unit group of rank one. Tran proved the conjecture for any quadratic extension of a complex quadratic field \cite{Tran2} and, along with Tian, for all cyclic cubic fields \cite{TranPeng1}. In these cases, the fields have unit group of rank one and two respectively.

In this paper we consider another class of number fields with unit group of rank two: totally imaginary cyclic sextic fields. This class of number fields poses its own set of challenges. To prove our main result we  develop new techniques which are found in Sections \ref{secbound} and \ref{counting1}. Using these methods we are able to prove:

\begin{theorem}\label{thmmain}	
	Let $F$ be an imaginary cyclic sextic field. Then the function $h^0$ on $\PicoF$ obtains its unique global maximum at the trivial class $[D_0]=[(O_F, 1)]$.
\end{theorem}

To prove Theorem \ref{thmmain}, we prove the equivalent statement \[h^0(O_F,1)>h^0(I, u)\text{ whenever }[(I, u)] \neq [(O_F, 1)].\] 
The proof strategy is outlined in Section \ref{road_map}. We consider two cases:
\begin{enumerate}
\item Section \ref{case1} proves the case where $I$ is not principal and is the shorter of the proofs. 
\item Sections \ref{case2} and \ref{case2d} provide the proof for principal $I$. The reason this is split over two sections is that the proof differs depending on the value of $\|\log u\|$.
\end{enumerate}

The size function $h^0$ is given by the logarithm of the sum
$$k^0(I,u) := \sum_{f \in I}e^{-\pi\|u f\|^2}.$$
To more fully understand this definition and its context, see Subsections \ref{Ara1}, \ref{Pic} and \ref{h0}. In order to prove Theorem \ref{thmmain} it is sufficient to show
\[k^0(O_F,1)> k^0(I,u)\text{ whenever }[(I, u)] \neq [(O_F, 1)].\]

To procure an upper bound on $k^0(I,u)$ in Sections \ref{case1} and \ref{case2}, we split it into four summands: 
$$ k^0(I,u) = 1 + \Sigma_1(I, u) + \Sigma_2(I, u)  + \Sigma_3(I, u) $$
where each sum 
$\Sigma_i$, $i \in\{1,2,3\}$ is taken over a set $S_i$ with $I \backslash\{0\}=S_1\cup S_2\cup S_3$. Specifically,
$$\Sigma_i(I,u) := \sum_{f \in S_i}e^{-\pi\|u f\|^2},$$ with
the sets $S_i$ being chosen strategically in a way that groups the elements of $I \backslash\{0\}$ based on the size of $\|uf\|^2$, as defined in Section \ref{counting1}. The set $S_1$ is chosen with the smallest values, $\|uf\|^2<6\cdot 2^{1/3}$, and the set $S_3$ has the largest values, with $\|uf\|\geq 6\cdot 3^{1/3}$. Theorem \ref{thmmain} is then proved in in Sections \ref{case1} and \ref{case2} by finding a sufficiently small upper bound for each summand by applying the results in Section \ref{counting1} and Corollary \ref{sum}.

Seeing this outline at this stage, while it is not fully explained, serves to help the reader understand why we prove results that depend on the value $\|uf\|^2$.

To this end, we also highlight the use of the bound $\|f\|^2<22$ which appears in Corollary \ref{sum} and at the beginning in Section \ref{secbound} as an assumed condition on the size of $f$ in several propositions and lemmas. These results are applied to the proof of Theorem \ref{thmmain} in Section \ref{case2d}. This is the case where $I$ is a principal ideal and $\|\log u\|<0.24163$. As a very high-level explanation, one may suspect that this case is more difficult because the class $[(I,u)]$ bears a lot of similarities to the trivial class $[(O_F, 1)]$ in that $I$ and $O_F$ are both principal and $u$ and 1 are, geometrically speaking, sufficiently close to one another. 

To prove that $k^0(O_F,1)> k^0(I,u)$ in Section \ref{case2d} we show that $k^0(I, u)-k^0(O_F,1)<0.$ As this difference may be very small, we instead prove \[\frac{k^0(I, u)-k^0(O_F,1)}{\|\log u\|^2} <0,\] since this fraction is larger in absolute value and easier to work with. If all nonzero elements $f\in O_F$ which are not roots of unity have the property that $\|f\|^2\geq 22$, we can  prove that this  quotient is negative by Corollary \ref{sum}. Otherwise we compute  this  quotient case by case using the results in Section \ref{secbound} (see the proof of Proposition \ref{sum3} and Table \ref{table1}).

Through trying different bounds on $\|f\|^2$ we were able to determine that $\|f\|^2<22$ was the smallest bound necessary in order to make the mathematics work out. 
 
Also, the assumption that $F$ is cyclic is vital. The Galois property allows us to make use of several invariance properties (see Lemmata \ref{lengf} and \ref{h0sym}) which are crucial in our proofs of Lemma \ref{G} and Propositions \ref{equiv} and \ref{taylor1}. Moreover, as $F$ is cyclic we can obtain an explicit description of the discriminant of $F$ (Lemma \ref{disF}) and the unit group $O_F^\times$ (Lemma \ref{unitF}). The cyclic property also implies that the log unit lattice of $F$ is hexagonal and allows for the efficient calculation of lower bounds on the lengths of elements of $O_F$, when viewed as a lattice in $\mathbb{R}^6$ (see Propositions \ref{minlength}, \ref{boundpd12}, \ref{boundpd3}, \ref{limpd1} and \ref{limpd2}).

All of the computer-aided computations in this paper are  straightforward; we only need to call a function either in  Mathematica or in Pari/gp to obtain the result. We use  Mathematica \cite{Mathematica} for the approximations in Section 2.7 and for calculating the upper and lower bounds in Sections 7.1, 7.2 and 8.4. We apply the LLL algorithm \cite{ref:1} and the function \textbf{qfminim()} in Pari/gp \cite{PARI2}, which utilizes the  Fincke-Pohst algorithm  \cite{ref:40} and enumerates all vectors of length bounded in a given lattice. These enumerations are used in the proofs of Propositions 4.3, 4.4, and 8.6.


\section{Preliminaries}\label{sec2}
\subsection{Notation}
Let $F$ be an  imaginary cyclic sextic field with maximal real subfield $K$ and imaginary quadratic subfield $k$. Then $K$ is a cyclic cubic field with the form  $K=\mathbb{Q}(\theta)$ for some integral element $\theta$. Further,  $k= \mathbb{Q}(\sqrt{-d})$ for some squarefree positive integer $d$, and  $F= K(\sqrt{-d})$. Thus we have the following setup:
\begin{center}
\begin{tikzpicture}
    \node (F) at (0,0) {$F=\mathbb{Q}\left(\theta,\sqrt{-d}\right)$};
    \node (B) at (-1,-2) {$K=\mathbb{Q}(\theta)$};
    \node (C) at (1,-4) {$k=\mathbb{Q}\left(\sqrt{-d}\right)$};
    \node (Q) at (0,-6) {$\mathbb{Q}$};
    \draw (F) -- (B); 
    \draw (C) -- (Q); 
    \draw (F) -- (C);
    \draw (B) -- (Q);
    
    \node (ZF) at (4,0) {$O_F$};
    \node (ZB) at (3,-2) {$O_K$};
    \node (ZC) at (5,-4) {$O_k$};
    \node (Z) at (4,-6) {$\mathbb{Z}$};
    \draw (ZF) -- (ZB); 
    \draw (ZC) -- (Z); 
    \draw (ZF) -- (ZC);
    \draw (ZB) -- (Z);
    
    \node (GF) at (8,0) {$\{1\}$};
    \node (GB) at (7,-2) {$G_k= \langle \varphi \rangle$};
    \node (GC) at (9,-4) {$G_K= \langle \sigma \rangle$};
    \node (GQ) at (8,-6) {$G=G_F = \langle \tau \rangle$ };
    \draw (GF) -- (GB); 
    \draw (GC) -- (GQ); 
    \draw (GF) -- (GC);
    \draw (GB) -- (GQ);    

\end{tikzpicture}

\end{center}

Here, $O_F, O_K$ and $O_k$ are the rings of integers and $G=\langle \tau \rangle, G_K= \langle \sigma \rangle$ and $G_k= \langle \varphi \rangle$ the Galois groups of $F, K$ and $k$ respectively.  Then $O_k=\mathbb{Z}[\delta]$ where 
 \[
   \delta = \left\{
     \begin{array}{ll}
       \sqrt{-d} & \text{if } d\equiv1,2 \mod 4\\
       \frac{1+\sqrt{-d}}{2} &  \text{ otherwise.}
     \end{array}
   \right.
\]

Observe that $\tau(\theta)=\sigma(\theta)$ and $\tau(\delta)=\varphi(\delta)$. We have the six embeddings $F\hookrightarrow\mathbb{C}$:
 \[\tau_1=\textbf{1}=\tau^0,\hspace{2em}  \tau_2 = \tau^1,\hspace{2em} \tau_3 = \tau^2,\]\[\overline{\tau_1} = \overline{\textbf{1}}=\tau^3,\hspace{2em} \overline{\tau_2} = \tau^4,\hspace{2em} \overline{\tau_3} = \tau^5.\]

In this paper, we use the map $\Phi: F \longrightarrow \mathbb{C}^3$ defined by 
$$\Phi(f) = (\tau_i(f))_{1 \le i \le 3} = (\tau_1(f), \tau_2(f), \tau_3(f)) \text{ for all } f \in F.$$
The length function on each $f\in F$ is given by
\[\|f\|^2:= \|\Phi(f) \|^2= 2\sum_{i=1 }^3 |\tau_i(f)|^2.\] 
For each $g \in K$, we also define
\[\|g\|_K^2:= \|(\sigma^i(g)_i) \|^2= \sum_{i=1 }^3 |\sigma^i(g)|^2,\] thus $\|g\|^2= 2\|g\|_K^2$, and:

\begin{lemma}\label{lengf}
	Let $f \in F$. Then $\|\tau_1(f) \| = \| \tau_2(f)\| = \|\tau_3(f)\|$.
\end{lemma}

The image $\Phi(I)$ of a fractional ideal $I$ of $F$  is a lattice in $\mathbb{C}^3$ and thus maps to  a lattice in $\mathbb{R}^6$ via $z \mapsto (\Re(z), \Im(z) )$ where $\Re(z)$ and $\Im(z)$ are the real and imaginary parts of $z$.

Let $p$ be the conductor of $K$ and $t=\gcd(p, d)$. The discriminant of $K$ is $\Delta_K = p^2$ and 
 \begin{displaymath}
   \Delta_k = \left\{
     \begin{array}{ll}
       -4d & \text{   if } d\equiv1,2\mod{4}\\
       -d & \text{   otherwise}.\\
     \end{array}
   \right.
\end{displaymath}

\begin{remark}\label{conductor}
	The conductor $p$ of $K$ has the form
$p = p_1 p_2 \cdots p_r,$
	where $r \in \mathbb{Z}_{>0}$ and $p_1, \cdots , p_r$ are distinct integers from the set
	$$ \{9\} \cup  \{q\mid q\text{ is prime }, q\equiv 1\mod 3\} = \{7, 9, 13, 19, 31, 37, \ldots \} \text{ (see \cite{Hasse}).}$$
	
\end{remark}


\subsection{The ring of integers $O_K$}\label{cubic}
We first recall the following result from  \cite{TranPeng1}.  \begin{proposition}\label{OK} 
There exists a $g \in O_K$ such that $\Tr(g) = g + \sigma(g) + \sigma^2(g) =0$ and one of the following holds:
\begin{enumerate}
\item[i)] $O_K = \mathbb{Z} \oplus \mathbb{Z}[\sigma] \cdot g$  or 
\item[ii)] $O_K \supset \mathbb{Z} \oplus \mathbb{Z}[\sigma] \cdot g$  and $[O_K : (\mathbb{Z} \oplus \mathbb{Z}[\sigma] \cdot g)] =3$.
\end{enumerate}

\end{proposition}

Using the proof from
 \cite[Proposition 2.3]{TranPeng1} in combination with the equation $\|g\|^2 =2\|g\|_K^2$ for $g \in K$ we conclude:

\begin{proposition}\label{minlength}
We have  $\|g\|_K^2 \ge \frac{2p}{3}$ and $\|g\|^2 \ge \frac{4p}{3}$ for all $g \in O_K \backslash \mathbb{Z}$. 
\end{proposition}

 Another structural observation of $O_K$ is:
\begin{lemma}\label{indtK}
	For any $f \in O_K\backslash \mathbb{Z}$, the set $\{1, f, \sigma(f)\}$ is $\mathbb{R}$-linearly independent.
\end{lemma}


\subsection{The unit lattice}\label{logunit}

We define the map $\log: F^\times \longrightarrow  \mathbb{R}^3$, the plane $\mathcal{H}$ and the log unit lattice $\Lambda$ as follows: 
$$\log(f):= (\log|\tau_i(f)|)_{1 \le i \le 3} \in \mathbb{R}^3 \text{ for all } f \in F^\times,$$
$$\mathcal{H} = \{ (v_1, v_2, v_3) \in \mathbb{R}^3: v_1 + v_2 + v_3 = 0 \},$$
$$\Lambda= \log(O_F^\times)= \{(\log|\tau_i(\varepsilon)|)_{i=1}^3: \varepsilon \in O_F^\times\}.$$
Here $\Lambda$ is a full rank lattice contained in  $\mathcal{H}$ by Dirichlet's unit theorem. Let $\mu_F$ be the set of roots of unity of $F$.
\begin{lemma}\label{unitF}
The unit group of $O_F$ is
	$O_F^{\times} = \mu_F  O_K^{\times}$.
\end{lemma}
\begin{proof}
The Hasse unit index of $F$ is $Q_F= [O_F^\times: \mu_F O_K^\times]$ by \cite{hasse52unitindex}. Since $k$ is an imaginary quadratic field its unit index $Q_k =1$. Both $k$ and $F$ are  totally complex and abelian with $[F:k]= 3$. Thus by \cite[Lemma 2 ]{hirabayashi88unitindex},  $Q_F = Q_k =1$. Therefore
$O_F^\times= \mu_F O_K^\times$.
\end{proof}

A lattice is called \textit{hexagonal} if it is isometric to the lattice $M \cdot \mathbb{Z}[\zeta_3]$ for some $M \in \mathbb{R}_{+}$ and a primitive cube root of unity $\zeta_3$.

\begin{corollary}\label{lambdahexan}
	The lattice $\Lambda$ is hexagonal.
\end{corollary}

\begin{proof}
	By Lemma \ref{unitF}, we have that $\Lambda = \log(O_F^{\times}) = \log(\mu_F \times O_K^{\times}) =\log( O_K^{\times})$. The result follows since $\log( O_K^{\times})$ is hexagonal by  \cite[Proposition 2.1]{TranPeng1}.
\end{proof}

Corollary \ref{lambdahexan} implies that $\Lambda $ has a $\mathbb{Z}$-basis given by two shortest vectors $b_1 = \log \varepsilon_1, b_2=\log \varepsilon_2$ for some $\varepsilon_1, \varepsilon_2 \in O_F^\times$ and with $\|b_1\| = \|b_2\|=\|b_2-b_1\|$ (Figure \ref{pic:hexagonal}). Let $\mathcal{F}$ be the fundamental domain of $\Lambda$ given by
$$\mathcal{F} = \left\{\alpha_1 \cdot b_1 + \alpha_2 \cdot b_2: \alpha_1, \alpha_2 \in \left(-\frac{1}{2}, \frac{1}{2}\right] \right\}.$$

\vspace{-0.5cm}
\begin{figure}[h]
	\centering
	\includegraphics[width=0.4\linewidth]{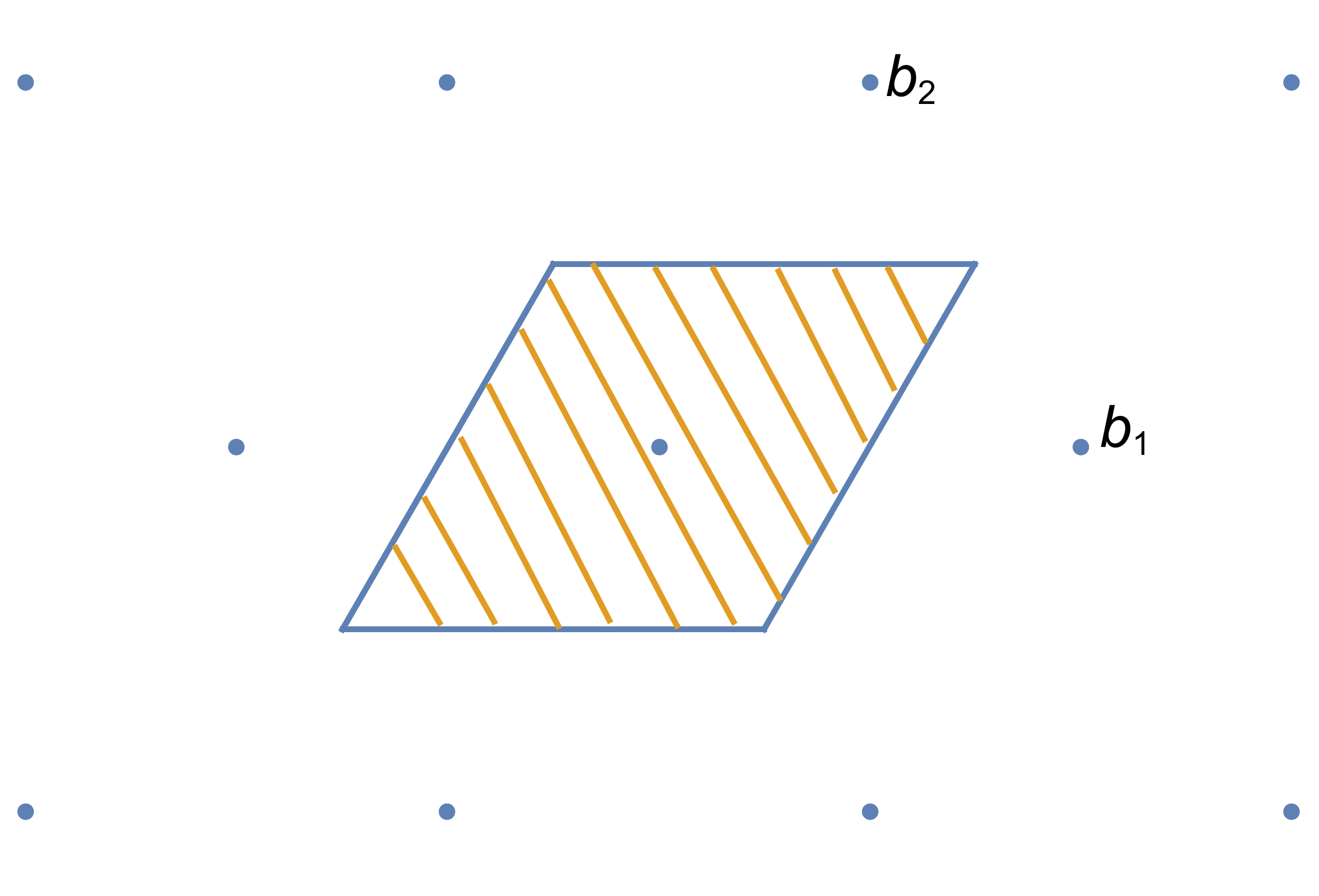}
	\caption{The lattice $\Lambda$ and $\mathcal{F}$ (the shaded area). \label{pic:hexagonal}}
\end{figure} 
\vspace{0.1cm}

\begin{remark}
We could also choose a different fundamental domain, such as the Voronoi domain. This would make the proof in Section \ref{case2a} slightly different   but the proofs in Sections \ref{case2c} and \ref{case2d} would  remain the same.
\end{remark}

We further define $\lambda$ to be the length of the shortest vectors of $\Lambda$, and $$B(w)= \{ \mathbf{x} \in O_F^{\times}: \|\log \mathbf{x} -w\| < \lambda\} \text{ for each } w \in \mathcal{F}.$$  
\begin{lemma}\label{Bomega}
Let $w \in \mathcal{F}$. Then $\# B(w) \le 4 \cdot (\# \mu_F)$. Moreover, 
$$B(w) \subset \{ 1, \mathbf{x}_1, \mathbf{x}_2, \mathbf{x}_3\}\cdot \mu_F \subset O_F^{\times} \text { where }$$ 
$$\|\log \mathbf{x}_1-w\| \ge  \sqrt{3}\lambda/4 , \|\log\mathbf{x}_2-w\| \ge  \lambda/2  \text{ and } \|\log \mathbf{x}_3-w\| \ge  \sqrt{3}\lambda/2.$$
\end{lemma}

\begin{proof}
	See the proof of  \cite[Lemma 2.2]{TranPeng1}, replacing $\pm 1$ with $\mu_F$.
\end{proof}

\begin{lemma}\label{lambda1}
If $p=7$ then  $\lambda \approx 1.44975$. Moreover,  $\lambda >1.83336$ when $p \ge 9$.
\end{lemma}

\begin{proof}
This lemma follows from the argument in the proof of  \cite[Lemma 2.3]{TranPeng1}, combined with Lemma \ref{unitF} and the fact that $\|g\|^2 =2\|g\|_K^2$ for $g \in K$.  
\end{proof}


\subsection{Arakelov divisors}\label{Ara1}

 \begin{definition}
 	An \textit{Arakelov divisor} of $F$ is a pair $D=(I,u)$ where $I$ is a fractional ideal of $F$ and $u$ is any element in $\mathbb{R}^3_{+}$. 
 \end{definition}
 
The Arakelov divisors of $F$ form the additive group $\text{Div}_F$. 
The \textit{degree} of a divisor $D = (I,u)$ is $\text{deg}(D): = \log{(N(u) N(I))}$, where the \textit{norm} of  $u = (u_1, u_2, u_3) \in \mathbb{R}^3$ is $N(u):= u_1 u_2 u_3$.  Define \[ u f:= u \cdot \Phi(f) = (u_i \cdot \tau_i(f)) \in \mathbb{C}^3\text{ for all }f \in I.\] Then
  $$\|u f\|^2 =\|u \cdot \Phi(f)\|^2 =2\sum_{i=1 }^3  u_i^2 \cdot |\tau_i(f)|^2.$$

Further, $ u I: =\{u f: f \in I\}$ is a lattice in $\mathbb{C}^3$. We call $u I$ the \textit{lattice associated} to $D$. Each element $f \in F^\times$ is attached to a \textit{principal Arakelov divisor} $(f):=(f^{-1} O_F, |f|)$. Here, 
$f^{-1} O_F$ is the principal ideal generated by $f^{-1}$, and $$|f|: =|\Phi(f)|= (|\tau^i(f)|)_{0 \le i \le 2} \in \mathbb{R}^3_{+}.$$ This divisor has degree $0$ by the product formula \cite{ref:4, ref:3}.

\subsection{The Arakelov class group}\label{Pic}
The Arakelov divisors of degree $0$ form a group $\DivoF$. 

\begin{definition}
The \textit{Arakelov class group }  $\PicoF$ is the quotient of $\DivoF$ by its subgroup of principal divisors.
\end{definition}
 
 This class group is similar to the Picard group of an algebraic curve. Define 
$\To = \mathcal{H} / \Lambda,$ a real torus of dimension $2$. Each class $v = (v_1, v_2, v_3) \in  \To$ can be embedded into $\PicoF$ by $v \mapsto D_v= (O_F, u)$ with $u = (e^{-v_{i}})_{i}$. Therefore, $\To$ can be viewed as a subgroup of $\PicoF$, and, by \cite[Proposition 2.2]{ref:4} we know:

\begin{proposition}\label{prop:structure1}
The map that sends the Arakelov class represented by a divisor $D = (I, u)$ to the ideal class of $I$ is a homomorphism from $ \PicoF$ to the class group $\cl_F$ of $F$. 
It induces the exact sequence 
\begin{align*}
 0 \longrightarrow \To \longrightarrow \PicoF \longrightarrow \cl_F \longrightarrow 0.
\end{align*}
\end{proposition}

The group $\To$ is the connected component of the identity of the topological group $\PicoF$. 
 Each class of Arakelov divisors in $\To$ is represented by a divisor $D = (O_F, u)$ for some 
$  u \in \mathbb{R}_{+}^3$, $N(u)=1$.  Here $u$ is unique up to multiplication by a unit in $ O_F^\times$ \cite[Section 6]{ref:4}.


\subsection{The function $h^0$}\label{h0}
 Let $D=(I,u)$ be an Arakelov divisor of $F$. 
Define
$$ h^0(D):=\log(k^0(D)),$$ 
$$ k^0(D) := \sum_{f \in I}e^{-\pi\|u f\|^2} = \sum_{ \mathsf{x} \in u I}e^{-\pi\| \mathsf{x} \|^2}.$$
 The function $h^0$ is  well-defined on $\PicoF$ and analogous to the dimension of the Riemann-Roch space $H^0(D)$ of a divisor $D$ on an algebraic curve\cite{ref:4, ref:3}. From \cite{TranPeng1} we have:
 
 \begin{lemma} \label{h0sym}
 	The function $h^0$ on $\To$ is invariant under the action of $\tau$. That is, 
 	$$h^0(D) = h^0(\tau(D)) \text{ for all } D \in \To.$$
 \end{lemma}

\begin{remark}\label{prindiv}
	Let $I$ be the principal ideal $I=fO_F$ for some $f \in F^{\times}$. Then
	$$D=(I,u)=(f O_F,u) = (f O_F, |f|^{-1})+ (O_F, u |f|)=(f^{-1})+ (O_F, u').$$
	Here $(f^{-1})$ is the principal Arakelov divisor generated by $f^{-1}$ and 
	$$u'= u |f|= (u_i |\sigma_i(f)|)_i \in \mathbb{R}^3_+.$$
	Thus $D$ and $D'=(O_F, u')$ are in the same class of divisors in $\text{Pic}^0_F$, and hence 
	$k^0(D) = k^0(D')$. Therefore, without loss of generality we can assume that $D$ has the form $(O_F, u)$ for some $u \in \mathbb{R}^3_+$ and $N(u)=1$. In other words, $[D] \in \T^0$.
\end{remark}



\subsection{Some estimates}\label{sec1}
 Let $L$ be a lattice in $\mathbb{R}^6$ and $\lambda$ the length of its shortest vectors. Using an argument similar to the proof of \cite[Lemma 3.2]{Tran2}, replacing $\pi$ with $\xi$, we have:

\begin{lemma}\label{err}
	For $M \geq \lambda^2\geq a^2>0$ and $\xi >0$,	
 		$$\sum_{\substack{ \mathsf{x}  \in L   \\ \|\mathsf{x} \|^2 \geq  M }}e^{-\xi \|\mathsf{x} \|^2 } \leq \xi  \int_{M}^{\infty} \!  \left( \left( \frac{2\sqrt{t}}{a} +1\right)^6 - \left( \frac{2\sqrt{M}}{a} -1\right)^6 \right) e^{- \xi  t}\, \mathrm{d}t.$$
\end{lemma}

The next result can be obtained by applying  Lemma \ref{err} with $a=\sqrt{6}$ and $\xi = \pi$.

\begin{corollary}\label{sum} 
	If $\lambda^2  \geq 6$, then 
	$$\sum_{\substack{ \mathsf{x} \in L   \\ \|\mathsf{x} \|^2 \geq 6\cdot 3^{1/3} }}e^{-\pi \|\mathsf{x} \|^2 } < 2.6049\cdot 10^{-9},\hspace{2em}\sum_{\substack{ \mathsf{x} \in L   \\ \|\mathsf{x} \|^2 \geq 22 }}e^{-(\pi-2/7) \|\mathsf{x}  \|^2 } <  10^{-23},$$
	$$ \sum_{\substack{ \mathsf{x} \in L   \\ \|\mathsf{x} \|^2 \geq 22 }}e^{-(\pi-2\sqrt{2}  \cdot 0.170856\hspace*{0.1cm}\pi-2/7) \|\mathsf{x}  \|^2 } < 2.19277\cdot 10^{-9}.\hspace{6em}$$
\end{corollary}

\section{Upper bounds for $p$ and $d$}\label{secbound}
In this section, we find upper bounds for $p$ and $d$ if there exists an element $f \in O_F\backslash \mu_F$ such that $\|f\|^2_F<22$. We restrict to this case, as these are the results needed for the proof of Theorem \ref{thmmain} in Section \ref{case2d}.

\begin{lemma}\label{disF}
	The discriminant of $F$ is $\Delta_F = \frac{p^4 \Delta_k^3}{\gcd(p,d)^2}$. Consequently, the index $$[O_F : O_K[\delta]] = \gcd(p,d) = t.$$	
\end{lemma}
\begin{proof}
		Observe that $O_K[\delta]= O_k \times O_K$ since $O_k=\mathbb{Z}[\delta]$ . The discriminant of the tensor product $ O_k \times O_K$ is $(p^2)^{[k:\mathbb{Q}]}  (\Delta_k)^{[K:\mathbb{Q}]} = p^4 \Delta_k^3$.

	By the conductor-discriminant formula, the discriminant of $F$ is equal to the product of the conductors of the characters of $F$\cite{artin1931}. The trivial character, the quadratic character and the two cubic characters have  conductors  $1$, $\Delta_k$, $p$, $p$  respectively. The two characters of order $6$ have conductor $\lcm(p,\Delta_k)$. Hence $\Delta_F= \Delta_k p^2 (\lcm(p,\Delta_k))^2 = \frac{ p^4 \Delta_k^3}{\gcd(p, \Delta_k)^2} = \frac{ p^4 \Delta_k^3}{\gcd(p, d)^2}$. The last equality is because $\Delta_k \in \{d, 4d\}$ and $p$ is odd by Remark \ref{conductor}. Thus the index of $O_k \times O_K$ inside $O_F$ is $\sqrt{\frac{p^4 \Delta_k^3}{\Delta_F}} = \gcd(p, d)$.
\end{proof}

Further, since $t = [O_F : O_K[\delta]]$  for every $f \in O_F$, we have $t f \in O_K[\delta]$. Hence, \[f = \frac{1}{t} (\gamma + \beta \delta) \text{ for some } \gamma, \beta  \in O_K.\]

\begin{proposition}\label{boundpd12}
	Let $d \equiv 1, 2 \mod 4$. Assume that $f =\frac{1}{t} (\gamma + \beta \delta) \in O_F\backslash(O_K \cup O_k \cup \mu_F)$  where $\gamma, \beta  \in O_K$ such that  $\|f\|^2<22$. Then we have the following.
	\begin{itemize}
		\item[i.] If $\beta  \in O_K\backslash\mathbb{Z}$, then $p \le 19$ and $d\le 22$.
		\item[ii.] If $\beta  \in \mathbb{Z}\backslash\{0\}$, then $p \le 61$ and $d\le 14$.
		
	\end{itemize}
\end{proposition}

\begin{proof}
	Assuming that $f$ satisfies the criterion in the statement of the proposition,
	\begin{equation}\label{eqlen221}
	22 > \|f\|^2 = 2 \frac{\|\gamma\|_K^2}{t^2} + 2 \frac{\|\beta\|_K^2 d}{t^2}.
	\end{equation}
	Let $\alpha \in O_K\backslash\mathbb{Z}$ be a shortest element. By the proof of \cite[Proposition 2.3]{TranPeng1},
	\begin{equation}\label{eqshortK}
	\|\alpha\|_K^2 \le \frac{2p +1}{3}. 
	\end{equation} 
	Also, recall that $\delta= \sqrt{-d}$. 
 
 We consider two cases:
	\begin{itemize} 
		\item[\textbf{Case i:}] $\beta \in O_K\backslash\mathbb{Z}$. 
		Note that $\tau^i|_K = \Re(\tau^i) =  \sigma^i, i=1, 2, 3$ and $\tau^3=\overline{\textbf{1}}$.
		The discriminant of the set $\mathcal{S}=\{1, \alpha, \tau(\alpha), \delta, f, \tau(f) \}\subset O_F$ is $\det(\mathcal{S})$, which is equal to:
				\[
		\left( \det(\tau^i(g))_{g \in \mathcal{S}, 0\le i \le 5}\right)^2 =   \frac{2^6 \left(\sqrt{d}\right)^6}{t^4}
		\begin{vmatrix} 
		1 & \alpha  & \sigma(\alpha)  \\ 
		1 & \sigma(\alpha)  & \sigma^2(\alpha)  \\
		1 & \sigma^2(\alpha)  & \alpha   
		\end{vmatrix}^2
		\cdot 
		\begin{vmatrix} 
		1 & \beta  & \sigma(\beta)  \\ 
		1 & \sigma(\beta)  & \sigma^2(\beta)  \\
		1 & \sigma^2(\beta)  & \beta  
		\end{vmatrix}^2.
		\]

	Since $\alpha, \beta  \in O_K\backslash\mathbb{Z}$, the two sets $\{1, \alpha, \sigma(\alpha)\}$ and $\{1, \beta,  \sigma(\beta)\}$  are  $\mathbb{R}$-linearly independent  (Lemma \ref{indtK}), so  $\det(\mathcal{S})\neq0$. Thus $\mathcal{S}$ is a set of independent elements in $O_F$, and, by Lemma \ref{disF},
	\begin{equation}\label{eqdis1}
	\det(\mathcal{S}) \ge |\Delta_F| =  \frac{p^4 \cdot (4d)^3}{t^2}.
	\end{equation}

	Combining this with Hadamard's inequality leads to
	\[
	\begin{tabular}{ll}
	$\det(\mathcal{S})$
	& $\le  \frac{ 2^6 \cdot  d^3}{t^4} \cdot  \|1\|_K^2 \cdot \|\sigma(\alpha)\|_K^2 \cdot \|\sigma^2(\alpha)\|_K^2 \cdot \|1\|_K^2 \cdot \|\sigma(\beta)\|_K^2 \cdot \|\sigma^2(\beta)\|_K^2$\\
	
	&$=2^6 \cdot  d \cdot  3^2 \cdot \|\alpha\|_K^4 \cdot  \left(\frac{ \|\beta\|_K^2 \cdot d} {t^2}\right)^2$.
	\end{tabular}
		\]
	From (\ref{eqlen221}) and (\ref{eqshortK}) we have
	\begin{eqnarray}	\label{eqdis2}
	\det(\mathcal{S})
	< 2^6 \cdot  d \cdot  3^2 \cdot \left(\frac{2p +1}{3}\right)^2 \cdot 11^2 = 2^6 \cdot  d \cdot  11^2 \cdot (2p +1)^2.
	\end{eqnarray}
	
	Applying (\ref{eqdis1}) and (\ref{eqdis2}) gives:
	\begin{equation}\label{eqboundpd1}
	\frac{p^2 \cdot d} {t} < 11 \cdot (2p +1).
	\end{equation}
	Since $t \le d$, this bound implies $p^2 < 11 \cdot (2p +1)$, which gives  $p \le 19$. We also have $t \le p$, then by (\ref{eqboundpd1}), $p\cdot d< 11 \cdot (2p +1)$. As $p \le 19$, it follows that $d \le 22$.\\

   \item[\textbf{Case ii:}] $\beta \in \mathbb{Z}$. Since $ f \notin O_k$ it follows that $\gamma \notin \mathbb{Z}$. Consequently, $\gamma \in O_K\backslash \mathbb{Z}$ and 
   \[\frac{2 \gamma}{t} = f + \overline{\textbf{1}}(f) \in O_F \cap K,\] 
   which implies $\frac{2 \gamma}{t} \in O_K\backslash \mathbb{Z}.$  By Proposition \ref{minlength}, $\|\frac{2 \gamma}{t} \|^2 \ge 4p/3$, and therefore
   $$\frac{4p}{3} \le \left\|\frac{2 \gamma}{t} \right\|^2 \le \|f\|^2 + \|\overline{\textbf{1}}(f)\|^2 < 22 + 22= 44.$$
   It follows that $p \le 61$. \\
   Now since $\frac{2 \beta \sqrt{-d}}{t} = f - \overline{\textbf{1}}(f) \in O_F \cap k$, we have $\frac{2 \beta \sqrt{-d}}{t} \in O_k = \mathbb{Z}[\sqrt{-d}]$. This implies $\frac{2 \beta }{t} = a \in \mathbb{Z}$. Note that $\beta \neq 0$ since $f \notin O_K$. By (\ref{eqlen221}),
   $$22> 2 \frac{\|\beta\|_K^2 d}{t^2} \ge \frac{6a^2d}{4} \ge \frac{3d}{2}.$$
   As a result,  $d \le 14$.
  \vspace*{-1em}
\end{itemize}
\end{proof}

\begin{proposition}\label{boundpd3}
	Let $d \equiv 3 \mod 4$. Assume that there exists \[f = \frac{1}{t} (\alpha + \beta \delta)\in O_F\backslash(O_K \cup O_k \cup \mu_F)\] where $ \alpha, \beta  \in O_K$ and $\|f\|^2<22$. Then:
	\begin{itemize}
		\item[i.] If $\beta  \in O_K\backslash\mathbb{Z}$, then $p \le 61$ and $d\le 59$.
		\item[ii.] If $\beta  \in \mathbb{Z}\backslash\{0\}$, then $p \le 61$ and $d\le 11$.
		
	\end{itemize}
\end{proposition}

\begin{proof}
	This is similar to the proof of Proposition \ref{boundpd12}, with $\delta= \frac{1+\sqrt{-d}}{2}$, $\Delta_F = \frac{p^4 d^3}{t^2}$, and   
	\begin{eqnarray}\label{eqlen222}
	22 > \|f\|^2 =  \frac{\|2\gamma + \beta\|_K^2}{2t^2} +  \frac{\|\beta\|_K^2 d}{2t^2}.
	\end{eqnarray}

	When $\beta  \in O_K\backslash\mathbb{Z}$, we obtain the inequality	
	\begin{equation}\label{eqboundpd2}
	\frac{p^2 \cdot d} {t} < 22 \sqrt{2} \cdot (2p +1).
	\end{equation}
	This implies that $p \le 61$ and $d \le 59$.\\
	When $\beta  \in \mathbb{Z}$, since $(2\gamma + \beta)/t = f + \overline{\textbf{1}}(f) \in O_F \cap K =O_K$ and this element is not in $\mathbb{Z}$, Proposition \ref{minlength} gives
		$$4p/3 \le \|(2\gamma + \beta)/t\|^2 \le \| f\|^2 + \|\overline{\textbf{1}}(f) \|^2 <44,$$
	which implies $p \le 61$. Now since 
	$\beta\sqrt{-d}/t = f - \overline{\textbf{1}}(f)  \in O_F \cap k =O_k = \mathbb{Z}[\frac{1+\sqrt{-d}}{2}]$, we have $\beta/t =a \in \mathbb{Z}$. Hence (\ref{eqlen222}) provides that $22>  \frac{d}{2}\|\frac{\beta}{t}\|_K^2 = \frac{3a^2d}{2} \ge \frac{3d}{2}$, and thus $d \le 11$.

\end{proof}

\begin{proposition}\label{bound}
	Assume there exists $f \in O_k\backslash\mathbb{Z}$ with $\|f\|^2<22$. Then $d \in \{1, 2, 3, 7, 11\}$.
\end{proposition}

\begin{proof}
Since $f\in O_k\backslash\mathbb{Z}$, $f= m + n \delta$ for some $ m, n \in \mathbb{Z}$, $n\ne 0$. 
If $d \equiv 1, 2 \mod 4$, $\|f\|^2= 6(m^2 + n^2 d)<22$, then $d \in \{1, 2\}$. When $d \equiv 3 \mod 4$, we have $\|f\|^2= 6\left( (m + n/2)^2 + n^2 d/4 \right)<22$, which implies $d \in \{3, 7, 11\}$.
\end{proof}

\begin{proposition}\label{bounp}
	Assume there exists $f \in O_K\backslash\mathbb{Z}$ with $\|f\|^2<22$. Then $p \in \{7, 9, 13\}$.
\end{proposition}

\begin{proof}
If $f \in O_K\backslash\mathbb{Z}$ then $\|f\|^2 \ge 4p/3$ by Proposition \ref{minlength}. The result then follows. 
\end{proof}

\begin{proposition}\label{limpd1}
	Let  $p\le 61$ and $d \equiv 1, 2 \mod 4$ with $d\le 22$. Assume that there exists  $f \in O_F\backslash(O_K \cup O_k \cup \mu_F)$ such that  $\|f\|^2<22$. Then $$(p, d) \in \{ (7, 1), (9,1), (13,1), (19,1), (7,2), (9,2),  (9,6), (13, 13), (7,14), (7, 21), (9,21)\}.$$
\end{proposition}
\begin{proof}
	We consider two cases and use a similar idea as in the proof of Proposition \ref{boundpd12}. In the first case, when $\beta \in O_K\backslash\mathbb{Z}$, we have that $p^2 d/t \le 11(2p +1)$. In the second case, $\beta \in \mathbb{Z}$. Here, since $2\|\gamma\|_K^2/t^2 + 2\|\beta\|_K^2 d/t^2 = \|f\|^2 <22$, one has the bound $1/t^2(4p/3 + 6d) <22$. Using these inequalities, we obtain the values for $(p,d)$.
\end{proof}

\begin{proposition}\label{limpd2}
	Let  $p\le 61$ and $d \equiv 3 \mod 4, d\le 59$. Assume  $f \in O_F\setminus(O_K \cup O_k \cup \mu_F)$ such that  $\|f\|^2<22$. Then the possible values for $(p, d)$ are:
\[(p,3)\text{ with }p\in\{7, 9, 13, 19, 31, 37, 43\},\]
\[(p,7)\text{ with }p\in\{7, 9, 13, 19, 31\},\text{ and}\]
\[(7, 11), (9,11),  (13, 11), (9,15), (19, 19), (31, 31), (7, 35), (9, 39), (13, 39), (43, 43), (9, 51).\]	
\end{proposition}
\begin{proof}
One has either $p^2 d/t \le 22 \sqrt{2} (2p +1)$ or $1/t^2(2p/3 + 3d) <22$, yielding the result.
\end{proof}

\section{Counting short elements}\label{counting1}
Given $u$, $I$ and  an Arakelov divisor $D = (I, u)$ of degree $0$, we split the set  $I\setminus\{0\}$ into three subsets, since each subset will be counted using different techniques: 
$$S_1(I,u)=\{f\in I\setminus\{0\} :\| u f\|^2 < 6\cdot 2^{1/3}\},$$
$$S_2(I,u)=\{f\in I\setminus\{0\} :6\cdot 2^{1/3} \le \| u f\|^2 < 6\cdot 3^{1/3}\}\text{ and }$$
$$S_3(I,u)=\{f\in I\setminus\{0\} :\|u f\|^2 \ge 6\cdot 3^{1/3}\}.$$

In this section, we determine an upper bound for the cardinality of the set of ``short" elements in $S_2(I,u)$.

For any $f \in S_2(I,u)$, one has $|N(uf)| <3$ since $\|u f\|^2 \ge 6\cdot |N(uf)|^{1/3}$. As the degree of $D$ is $0$, $N(u)\cdot N(I) =1$. Therefore $|N(f)|/N(I) = |N(uf)| <3$, and  $|N(f)|/N(I) \in \{ 1, 2\}$. We will split $S_2(I,u)$ into two sets according to whether $|N(f)|/N(I)$ is 1 or 2:
\begin{equation}\label{numBi12}
S_{2, i}(I,u)=  \{ f \in I: 6\cdot 2^{1/3} \le \|uf  \|^2 < 6\cdot 3^{1/3} \text{ and } |N(f)|/N(I)=i \}.\\ 
\end{equation}
Then,  $$S_2(I,u) = S_{2,1} (I,u) \cup S_{2,2} (I,u).$$

\begin{lemma}\label{unitB2}
	If $O_K$ has a prime ideal of norm 2 and there exists $\epsilon \in O_F^{\times}$ with $\|\epsilon\|^2<81$, then $\epsilon \in \mu_F$.
\end{lemma}

\begin{proof}
	For the sake of contradiction, assume $\epsilon \notin \mu_F$. By Lemma \ref{unitF}, $\epsilon = \zeta \cdot \epsilon'$ for some $\zeta \in \mu_F$ and some $\epsilon' \in O_K^{\times}$. Now $81 > \|\epsilon\|^2  = \|\epsilon'\|^2 \ge 4p/3$ by Proposition  \ref{minlength}. Thus $p \in \{ 7, 9, 13, 19, 31, 37, 43\}$. If $O_K$ has an ideal of norm 2 then $p\in\{31, 43\}.$ If $p =31$, the regulator of $K$ is $R_K \approx 12.196$. Since $\Lambda$ is hexagonal (Remark \ref{lambdahexan}), $\|\log(\epsilon)\|^2 \ge 2 R_K \approx 24.392$. This leads to $\|\epsilon\|^2 \ge 225.615$, contradicting the condition $\|\epsilon\|^2<81$. Similarly, if $p=43$, $R_K \approx 18.9218$. This leads to a contradiction, as $\|\epsilon\|^2 \ge 607.392$.
\end{proof}

\begin{proposition}\label{boundm2} 
	Assume that $N(u) = 1/N(I)$. 
	\begin{itemize}
		\item [i.] If $p<31$, then  $\#S_{2,2}(I,u) = 0$.
		\item [ii.] If $p\ge31$, then  $\#S_{2,2}(I,u) \le 6 \cdot (\#\mu_F )$. 
	\end{itemize} 
\end{proposition}
\begin{proof}
	Let $m_2 = \#S_{2,2}(I,u)$. 
	Since $|N(f)|/N(I) = |N(uf)| =2$ for all $ f \in S_{2,2}$, $fO_F = P I$ for some ideal $P$ in $O_F$ with $N(P) =2$. That is, each $f \in S_{2,2}$ corresponds to a prime ideal of norm 2 of $O_F$. If $m_2>0$, $O_F$ has a prime ideal of norm 2 and so does $O_K$. 
	
	\begin{itemize}
		\item[i.] If $p \in \{ 7, 9, 13, 19\}$, then $O_K$ has no ideals of norm 2. It means $m_2=0$.
		\item[ii.] If $p\ge 31$, then we have at most 6 distinct ideals of norm 2. Hence there are $m_2/6$ elements of $S_{2,2}$ which correspond to the same ideal of norm 2. Each of these elements must differ (pairwise) by a multiple of a unit. Thus there are $m_2/6$ distinct units; denote one of them by $\epsilon$. Then $\epsilon = f g^{-1}$ for some $f,g \in S_{2,2}$, and
	\[
	\begin{tabular}{ll}
	$\|\epsilon\|^2 = \|f g^{-1}\|^2 $ & $= 2\left(\frac{|u_1 f|^2}{|u_1 g|^2} + \frac{|u_2 \tau_2(f)|^2}{|u_2 \tau_2(g)|^2} + \frac{|u_3 \tau_3(f)|^2}{|u_3 \tau_3(g)|^2} \right)$\\
	& $\le 2\left(|u_1 f|^2+ |u_2 \tau_2(f)|^2 + |u_3 \tau_3(f)|^2\right) \left(\frac{1}{|u_1 g|^2}+ \frac{1}{|u_2 \tau_2(g)|^2} + \frac{1}{|u_3 \tau_3(g)|^2}\right)$ \\
	& $\le \|uf\|^2 \cdot \frac{\|u g\|^4}{4 \cdot |N(ug)|} < 6\cdot 3^{1/3} \cdot \frac{(6\cdot 3^{1/3})^2}{8} = 81$.
	\end{tabular}
	\]
		If $m_2>0$, then the above bound provides that those units are roots of unity by Lemma \ref{unitB2}. Hence $m_2/6 \le \#\mu_F$ and the result follows. 
	\end{itemize}
	\vspace{-1em}
\end{proof}

\begin{proposition}\label{boundm1} 
	Assume that $N(u) = 1$ and $m_1=\#S_{2,1}(O_F,u)$. Then,
	\begin{itemize}
		\item[i.]  if $p<31$, then $m_1  \le 19 \cdot (\#\mu_F )$, and 
		\item[ii.]  if $p \ge 31$, then $m_1 \le \#\mu_F$.
	\end{itemize} 
\end{proposition}

\begin{proof}%
 For all $f \in S_{2,1}(O_F,u)$, $f \in  O_F$ and $|N(f)|=N(O_F)=1$. That means all elements in $S_{2,1}(O_F,u)$ are units. Let $\epsilon = f g^{-1}$ for two distinct elements $f,g \in S_{2,1}$. Then
	\[
	\begin{tabular}{ll}
	$\|\epsilon\|^2 = \|f g^{-1}\|^2$ & $= 2\left(\frac{|u_1 f|^2}{|u_1 g|^2} + \frac{|u_2 \tau_2(f)|^2}{|u_2 \tau_2(g)|^2} + \frac{|u_3 \tau_3(f)|^2}{|u_3 \tau_3(g)|^2} \right)$\\
& $= \|uf\|^2 \cdot\left(\frac{1}{|u_1 g|^2}+ \frac{1}{|u_2 \tau_3(g)|^2} + \frac{1}{|u_3 \tau_3(g)|^2}\right)$ \\
& $\le 6\cdot 3^{1/3} \cdot 5.15519  	\approx 44.61$.	
	\end{tabular}
		\]
		
	The last inequality is obtained because if $x_1 \cdot x_2 \cdot x_3 =1$ and $6\cdot 2^{1/3} \le 2 \cdot (x_1^2+x_2^2+x_3^2) < 6\cdot 3^{1/3}${ then}  \[\frac{1}{x_1^2}+ \frac{1}{x_2^2} + \frac{1}{x_3^2} < 5.15519.\]
	
Hence  there are  $m_1$ distinct units $\epsilon$ with $\|\epsilon\|^2  \le 44.61$. Assume that there exists such an $\epsilon$ where $\epsilon \notin \mu_F$. By Lemma \ref{unitF}, $\epsilon = \zeta \cdot \epsilon'$ for some $\zeta \in \mu_F$ and some $\epsilon' \in O_K^{\times}\backslash \{\pm 1\}$. Using Proposition \ref{minlength}, $44.61 > \|\epsilon\|^2  = \|\epsilon'\|^2 \ge 4p/3$. Thus $p \in \{ 7, 9, 13, 19, 31\}$, and:
 \begin{itemize}
 	\item If $p \ge 37$, then all $m_1$ units $\epsilon$ belong to $\mu_F$. It follows that $m_1  \le \#\mu_F$.
 	\item If $p \in \{ 7,  9, 13, 19, 31\}$, denote by $m_1'$ the number of $\epsilon'$ up to sign, then 
 	$$m_1 = \text{ the number of such } \epsilon \text{  in }  \mu_F\ + \text{ the number of such } \epsilon \text{ not in }  \mu_F.$$
 	Therefore and upper bound for $m_1$ is
 	\begin{equation}\label{eqm1}
 	m_1 \le  \#\mu_F +  (\#\mu_F ) \cdot m_1'  = (\#\mu_F ) \cdot (m_1'+1).
 	\end{equation}

 	We can find all $\epsilon' \in O_K\backslash \{\pm1\}$ up to sign for which  $\|\epsilon'\|_K^2 < 44.61/2$ (equivalently $\|\epsilon'\|^2 < 44.61$), and $ |N_K(\epsilon')| = 1$ using an LLL-reduced basis \cite[Section 12]{ref:1} of the lattice $O_K$  or by applying the Fincke--Pohst algorithm  \cite[Algorithm 2.12]{ref:40}:
 		\begin{center} 	
 	\begin{tabular}{|c| c |c|c|c|c|}
 		\hline
 		$p$ & 7 & 9 & 13 & 19 & 31 \\ 
 		\hline 
 		$m_1' \le $ & 18 & 12 & 6 & 6 & 0\\     
 		\hline 
 		
 	\end{tabular}\\ 
 	\end{center}
 The result then follows by the bound for $m_1$ in (\ref{eqm1}).
 \end{itemize}
 \vspace{-1em}
\end{proof}

\begin{proposition}\label{boundm1p=7} 
	Assume that $N(u) = 1$, $m_1=\#S_{2,1}(O_F,u)$ and $ \|u\|^2 \le 6.4653$. 
	
\begin{itemize}
		\item [i.] If $p<31$, then  $m_1 \le 12 \cdot (\#\mu_F )$, and
		\item [ii.] if $p\ge31$, then  $m_1 =0$. 
\end{itemize} 
	
\end{proposition}

\begin{proof}
First, we see that $S_{2,1}(O_F,u) \cap \mu_F = \emptyset$. This is because if there exists $f \in S_{2,1}(O_F,u) \cap \mu_F$, then $\|uf\|^2 = \|u\|^2 \le 6.4653 < 6 \cdot 2^{1/3}$ which is a contradiction. 

 As all $f\in S_{2,1}(O_F,u)$ are units, we bound $\|f^{-1}\|$ as in  the proof of Proposition \ref{boundm1}:
 	\[
	\begin{tabular}{ll}
	$\|f^{-1}\|^2 = \|u/uf\|^2 $ & $= 2\left(\frac{|u_1 |^2}{|u_1 f|^2} + \frac{|u_2 |^2}{|u_2 \tau_2(f)|^2} + \frac{|u_3|^2}{|u_3 \tau_3(f)|^2} \right)$\\
	& $= \|u\|^2 \cdot\left(\frac{1}{|u_1 f|^2}+ \frac{1}{|u_2 \tau_2(f)|^2} + \frac{1}{|u_3 \tau_3(f)|^2}\right)$ \\
	& $\le 6.4653 \cdot 5.15519 
	\approx 	33.33$.
	\end{tabular}
		\]
	
	Thus there are $m_1$ units with squared length $\le 33.33$.  
 Similar to the proof of Proposition \ref{boundm1}, and using the fact that $S_{2,1}(O_F,u) \cap \mu_F = \emptyset$, we have $m_1=  (\#\mu_F ) \cdot m_1'$ where 
	 \[m_1'= \#\{\epsilon' \in O_K^{\times}\backslash \{\pm1\}:\|\epsilon'\|_K^2\le 33.33/2\}. \]
	 
	When $p \ge 31$, then $m_1'= 0$ by Proposition \ref{minlength}. When $p < 31$, we compute the numbers  $m_1'$ and find:\\
	 
	 \begin{tabular}{|c| c |c|c|c| }
	 	\hline
	 	$p$ & 7 & 9 & 13 & 19   \\ 
	 	\hline 
	 	$m_1' \le $ & 12 & 6 & 6 & 3  \\     
	 	\hline 
	 \end{tabular}\\

 In these cases, $m_1' \le 12 $.  Hence $m_1 \le 12 \cdot (\#\mu_F)$.
\end{proof}

\section{Road map for the proof of Theorem \ref{thmmain}} \label{road_map}

In this section we give a road map of how we prove Theorem \ref{thmmain}. This proof requires us to consider several cases which we outline below. We seek to prove: \[h^0(O_F,1)>h^0(I, u)\text{ whenever }[(I, u)] \neq [(O_F, 1)].\] The case where $I$ is not principal is proved in Section \ref{case1} and is the shorter of the proofs. Sections \ref{case2} and \ref{case2d} prove the theorem in the case where $I$ is principal.

For an (Arakelov) divisor $D=(I, u)$, recall that $k^0(D) := \sum_{f \in I}e^{-\pi\|u f\|^2}$. Also recall $D_0:=(O_F,1)$. Since $h^0(D)=\log(k^0(D)),$ it is sufficient to prove: \[k^0(D)<k^0(D_0)\text{ whenever }[D] \ne [D_0].\] 

We split $k^0(D)$ into four summands: 
 \begin{equation} \label{eqsum}
\begin{split}
k^0(D) &  = 1 + \Sigma_1(I, u) + \Sigma_2(I, u)  + \Sigma_3(I, u) \text{, where } \\ 
   \Sigma_{i}(I, u) &= \sum_{f\in S_i(I,u)}e^{-\pi \|u f\|^2 },\ i \in\{1,2,3\}.
\end{split}
\end{equation}

In previous papers on the size function for number fields \cite{ref:14,ref:15, Tran2,TranPeng1}, $k^0(D)$ was split into three summands which were then bounded to conclude $k^0(D)<k^0(D_0)$. The proof in this paper is more technical and requires four summands to find a sufficiently tight upper bound on $k^0(D)$. 
We bound them as follows:
\begin{itemize}
\item $\Sigma_1(I,u)$: We bound this sum twice, in Sections \ref{case1} and \ref{case2}, obtaining different results depending whether $I$ is principal.

\item $\Sigma_2(I,u)$: This is bounded using results from Section \ref{counting1}. Establishing this bound for is very different than techniques used in previous papers.

\item $\Sigma_3(I,u)$: Bounding $\Sigma_3(I,u)$ is accomplished by applying Corollary \ref{sum}.

\end{itemize}

\subsection{Strategy for Section \ref{case1}: $I$ is not principal}

We prove that $\Sigma_1(I,u)=0$ thus getting a small upper bound on $k^0(D)$. The result follows quickly.

\subsection{Strategy for Sections \ref{case2} and \ref{case2d}: $I$ is principal}

By Remark \ref{prindiv}, when $I$ is principal, we can assume that the class of divisor $[D]$ has the form  $[(O_F,u)]$,  for some $u \in \mathbb{R}^3_+$ and $N(u)=1$. 
With the notation from Section \ref{sec2}, the vector $u$ can be chosen such that  $w =-\log{u} \in \mathcal{F}$. It leads to  
$w= \alpha_1 \cdot b_1 + \alpha_2 \cdot b_2 \text{ for some } \alpha_1, \alpha_2 \in \left(-\frac{1}{2}, \frac{1}{2}\right].$\\
To establish Theorem \ref{thmmain} we divide this case into subcases depending on $\|w\|$. When $\|w\|$ is sufficiently large we can bound $\Sigma_1(O_F,u)$ to obtain the result that $k^0(D)<k^0(D_0)$ via Proposition \ref{S1}. We use this method in Section \ref{case2}, which considers the case where $ \|w\| \ge 0.24163$. This is divided into two separate subcases  \ref{case2a} and \ref{case2c} depending on the value of $\|w\|$, but the strategy remains similar for both cases.

 Finally, in Section \ref{case2d} we consider values of $w$ with $0< \|w\| < 0.24163$. Geometrically speaking, this is the case where $u$ is close to 1, so that $k^0(D)$ and $k^0(D_0)$ are very close in value, which makes this case more difficult than the others. We cannot obtain a useful bound on $\Sigma_1(O_F,u)$ and we must approach this case differently than the others.  Here we use a technique called ``amplification" \footnote{We thank Ren\'e Schoof for introducing this technique to us.} as follow.  Instead of proving that $k^0(D)-k^0(D_0)<0$, we consider the quantity
 \[[k^0(D)-k^0(D_0)]/\|w\|^2\]
 and prove it is negative. We divide by $\| w\|^2$ because $k^0(D)-k^0(D_0)$ may be extremely small, and this division scales it up to a value that is more tractable to bound.
 
We split this quantity into three separate sums, 
 $$[k^0(D)-k^0(D_0)]/\|w\|^2= T_1(u) + T_2(u) + T_3(u)$$
 where $T_i(u)$, $i\in\{1,2,3\}$ are defined at the beginning of Section \ref{case2d}, and we prove that $ T_1(u)+T_2(u)+T_3(u)<0$. 
 The definition of these $T_i(u)$ values takes several lines to develop, hence we will not define them here. We will only note that these are defined differently than the $\Sigma_i(I,u)$ values used in the previous two sections. Thus the proofs in Section \ref{case2d} rely on different techniques to establish, including the Galois properties of the fields. The bound for $T_3(u)$ uses the most innovative techniques and relies on Section \ref{secbound}. We prove Theorem \ref{thmmain} directly by applying Propositions \ref{equiv} and \ref{Gat1}--\ref{sum3}.


\section{Proof of Theorem \ref{thmmain} when $I$ is not principal}\label{case1}
As $I$ is not principal, $|N(f)|/N(I)\geq 2$ for all  $f \in  I \backslash \{0\}$. 
	 We recall that $N(I) N(u) =1$ since $\deg(D) = 0$. As a consequence, 
	$$\|u f\|^2 \geq 6 |N(u f)|^{2/6} = 6 |N(u) N( f)|^{1/3} = 6 \left(\frac{|N(f)|}{N(I)}\right)^{1/3} \geq 6\cdot 2^{1/3}.$$
	This inequality holds for any nonzero $f \in I$. Therefore, the length of the shortest vectors of the lattice $uI$ is $ \lambda \geq 6\cdot 2^{1/3}$. This implies that $\Sigma_1(I,u) = 0$ since $S_1(I, u)=\emptyset$, and that $\Sigma_{3}(I,u) < 2.6049\cdot 10^{-9}$ by Corollary \ref{sum}. 

We now show that $\Sigma_{2} (I, u) \le  6 \cdot (\#\mu_F )\cdot e^{-6\cdot 2^{1/3} \pi}$. 
It  is sufficient to find an upper bound for $\#S_{2 }(I,u)$.  First, we show that $S_{2,1}(I,u)  = \emptyset$: This is because if it contains some $f \in I$, then $fO_F =  I$, which contradicts the fact that I is not principal.  Hence $\#S_{2 }(I,u) = \#S_{2,2}(I,u)$ by (\ref{numBi12}).  By Proposition \ref{boundm2}, one has $\#S_{2,2}(I,u) \le 6 \cdot (\#\mu_F )$. The upper bound for $\Sigma_{2} (I, u)$ is implied. It follows that 
\[k^0(D) = \Sigma_{1} (I, u) + \Sigma_{2} (I, u) + \Sigma_{3} (I, u)< 1+ 6 \cdot (\#\mu_F )\cdot e^{-6\cdot 2^{1/3} \pi} +2.6049\cdot 10^{-9}.\] 
It is obvious that  $k^0(D_0) > 1 + (\#\mu_F ) \cdot e^{-6 \pi}$. Therefore  $k^0(D) < k^0(D_0)$ and Theorem \ref{thmmain} is proved when $I$ is not principal.

\section{Proof of Theorem \ref{thmmain} when $I$ is principal and $ \|w\| \ge 0.24163$}\label{case2}

\begin{proposition}\label{S1}
	Let $D=(O_F, u)$ be a divisor of degree 0. Then $k^0(D_0)> k^0(D)$ if one of the following conditions hold:
	\begin{itemize}
		\item[i.]  $\Sigma_1(O_F,u)<(\#\mu_F ) \cdot 4.28\cdot 10^{-9}$.
		
		\item[ii.] $ \|u\|^2 \le 6.4653$ and $\Sigma_1(O_F,u)<(\#\mu_F ) \cdot 4.62\cdot 10^{-9}$. 
	\end{itemize}	
\end{proposition}

\begin{proof}
As $\deg(D) = 0$, one has $N(u)=1$. For $f \in O_F\setminus\{0\}$, $|N(f)|\ge N(O_F)=1$, hence 	$$\|u f\|^2 \geq 6 |N(u f)|^{1/3} \ge 6 |N(u) N( f)|^{1/3} = 6 N(u)^{1/3} |N( f)|^{1/3} \geq 6.$$
	Therefore, the length of the shortest vectors of the lattice $uO_F$ is $ \lambda \geq \sqrt{6}$. 
	By Corollary \ref{sum},
	$\Sigma_{3}(O_F,u) < 2.6049\cdot 10^{-9}$.

We have $\#S_{2}(O_F,u) = \#S_{2, 1}(O_F,u) +\#S_{2, 2}(O_F,u)$. 
By Propositions \ref{boundm2} and \ref{boundm1}, $\#S_{2}(O_F,u)  \le 19 \cdot (\#\mu_F )$.	As a consequence,  
	$$\Sigma_{2}(O_F, u) \le  (\#S_{2}(O_F,u)) \cdot e^{-6\cdot 2^{1/3} \pi} \le 19 \cdot (\#\mu_F )\cdot e^{-6\cdot 2^{1/3} \pi}.$$ 
	
	Substituting this into $k^0(D) =1+ \Sigma_1(O_F,u) + \Sigma_2(O_F,u) +\Sigma_3(O_F,u)$, we have
	$$k^0(D) <  1+\Sigma_1(O_F,u)+ 19 \cdot(\#\mu_F )\cdot e^{-6\cdot 2^{1/3} \pi}+2.6049\cdot 10^{-9}.$$
Since $k^0(D_0) > 1 + (\#\mu_F ) \cdot e^{-6 \pi}$, then to show $k^0(D_0)> k^0(D)$, it is sufficient to prove
	\begin{equation*}
		\Sigma_1(O_F, u)  < (\#\mu_F )  \cdot\left( e^{-6 \pi} -  19 \cdot e^{-6\cdot 2^{1/3} \pi}\right) -2.6049\cdot 10^{-9}.
	\end{equation*} 
	The right hand side is greater than $ (\#\mu_F ) \cdot 4.28\cdot 10^{-9}$ because $\#\mu_F \ge 2$. Thus, the first statement i. is proved.

	If $\|u\|^2 \le 6.4653$, then Propositions \ref{boundm2} and \ref{boundm1p=7} imply that $\#S_{2}(O_F,u)  \le 12 \cdot (\#\mu_F )$. The statement ii. is then proved by using a similar argument.
\end{proof}

Proposition \ref{S1} is essential in proving the main theorem as showed below.

\begin{remark}
To prove Theorem \ref{thmmain}, it is sufficient to show $\Sigma_1 (O_F, u) <(\#\mu_F ) \cdot 4.28\cdot 10^{-9}$ for all $w=(x, y, z) \ne (0,0,0)$ 
\end{remark}

\begin{lemma}\label{Bomega2}
	Recall 
	$S_1(O_F, u)= \{f \in O_F^{\times}: \|u f\|^2 < 6\cdot 2^{1/3}\}.$ For each $f \in O_F$, define $ v_f = \log f$.  Then $v_f  \in \Lambda$ for each $f \in S_1(O_F, u)$. 
\end{lemma}
\begin{proof}
Assume $f \in O_F \setminus \{0\}$ and $\|u f\|^2 < 6\cdot 2^{1/3}$. Since $N(u)=1$, $N(u f) = N(f)$.
Thus
$$6\cdot 2^{1/3} > \|u f\|^2 \geq 6 |N(u f)|^{1/3} = 6 |N(f)|^{1/3}.$$
This implies that $|N(f)|=1$. That is, $f \in O_F^{\times}$. 
\end{proof}

 Let $w \in \mathcal{F}$. We consider two subcases in the next two subsections. As $w\in\mathcal{F}$  $\|w\| \le \sqrt{3} \lambda/2$, where, from Lemma \ref{lambda1}, one has $\lambda \ge 1.44975$.  

\vspace*{0.3cm}
\subsection{Case  $0.324096\cdot \sqrt{2} < \|w\| \le \sqrt{3} \lambda_1/2$ }\label{case2a}
As \[0.324096\sqrt{2}<\|w\|=\|-\log u\|,\]
$\|u\|^2 \ge 6.38985.$
 Let $ f \in S_1(O_F, u)$. Then by Lemma \ref{Bomega2}, $v_f \in \Lambda$. It follows that \[\| \log (uf)\| = \| \log f + \log u \|= \|v_f - w\|,\]
and $\|v_f - w\| < \lambda_1$  since otherwise $\| u f\|^2 > 6\cdot 2^{1/3}$. Hence  $ f \in B(w)$. Therefore $S_1(O_F, u) \subset B(w)$. 
By Lemma \ref{Bomega}, 
$S_1(O_F, u)  \subset \{ 1, \mathbf{x}_1, \mathbf{x}_2, \mathbf{x}_3\}\cdot \mu_F \subset O_F^{\times}  $, where 
$$\|\log \mathbf{x}_1-w\| \ge 0.62776,\ 
\|\log\mathbf{x}_2-w\| \ge 0.72487 
\text{ and } \|\log \mathbf{x}_3-w\| \ge   1.25552.$$
Since $\log \mathbf{x}_i-w = \log(u \mathbf{x}_i)$ for $1 \le i \le 3$, we obtain that
$$\| u \mathbf{x}_1\|^2 \ge  6.71608,\
 \| u \mathbf{x}_2\|^2 \ge 6.94478 
 \text{ and } \| u \mathbf{x}_3\|^2 \ge 8.72718.$$
Thus 
$\| u \mathbf{x}_3\|^2>6\cdot 2^{1/3}$, which implies $\mathbf{x}_3 \notin S_1(O_F, u)$ and $S_1(O_F, u) \subset  \{ 1, \mathbf{x}_1, \mathbf{x}_2\}\cdot \mu_F$. 

Then Theorem \ref{thmmain} is proved by Proposition \ref{S1}(i) and the inequality:

\begin{center}
\begin{tabular}{ll}
$\Sigma_1(O_F,u)$ &$\le   \sum_{f \in \{ 1, \mathbf{x}_1, \mathbf{x}_2\}\cdot \mu_F}e^{-\pi \|u f\|^2 }$ \\
&$= (\#\mu_F) \cdot \left(e^{-\pi \|u \|^2}+e^{-\pi \|u\mathbf{x}_1\|^2}+e^{-\pi \|u\mathbf{x}_2\|^2}\right)$ \\
&$\le  (\#\mu_F) \cdot \left( e^{- 6.38985 \pi } 
+ e^{- 6.71608 \pi } 
+ e^{- 6.94478 \pi }\right)$\\  
&$\approx (\#\mu_F ) \cdot 2.5\cdot 10^{-9}$.
\end{tabular}
\end{center}

\vspace*{0.3cm}
\subsection{Case  $0.24163 \le \|w\| \le  0.324096 \cdot \sqrt{2}$}\label{case2c}
The bounds on $\|w\|$ give $6.11188\le \|u\|^2\le 6.4653$.
 For $f \in S_1(O_F,u)\backslash \mu_F$, one has $0 \ne v_f \in \Lambda$ by Lemma \ref{Bomega2} and 
$$\| \log (u f)\| = \| \log f + \log u\|= \|v_f - w\| \ge | \|v_f\| - \|w \|| \ge \lambda_1 -0.324096\cdot \sqrt{2} \ge 0.9914.$$
It follows that  $\|u f\|^2 \ge 7.7265 > 6\cdot 2^{1/3}$. Thus, $S_1(O_F,u) \subset   \mu_F$ and
$$\Sigma_1(O_F,u) \le   \sum_{f \in \mu_F  } e^{-\pi \|u f\|^2 } = (\#\mu_F) \cdot  e^{-\pi \|u \|^2}
\le (\#\mu_F) \cdot  e^{-6.11188\pi } \approx (\#\mu_F ) \cdot  4.582\cdot 10^{-9}.$$
Theorem \ref{thmmain} is then established in this case using Proposition \ref{S1}(ii). 

\section{Proof of Theorem \ref{thmmain} when $I$ is principal and $0< \|w\| <0.24163 $}\label{case2d}

We first fix the following notation, which will be used but not re-stated in lemmas and propositions throughout this section: Given any $u\in\mathbb{R}^3_+$, let $x,y,z\in\mathbb{R}$ be such that $u= (e^x, e^y, e^z)$. Then $w = -\log u =(-x, -y, -z) \in \mathbb{R}^3$ and $x+ y + z = 0$. 

For any $f \in O_F$, we define $f_i = |\tau_i(f)|$, $ i \in\{1, 2, 3\}$. Then 
   $$\|u f\|^2=2\left( e^{2 x} |\tau_1(f)|^2 + e^{2 y} |\tau_2(f)|^2 + e^{2 z} |\tau_3(f)|^2\right) =  2\left(f_1^2 e^{2 x} + f_2^2 e^{2 y}  +  f_3^2 e^{2 z}\right).$$
    For $f\in O_F$ we now define
\begin{equation*}\label{eqG}
  G(u,f) = e^{-\pi \|f\|^2} G_2(f,u)/\|w\|^2, 
 \end{equation*}
 where
 \begin{center}
\begin{tabular}{ll}
$G_1(u,f)$ & $=e^{-\pi[ \|u f\|^2 - \|f\|^2]} - 1$= $e^{-2\pi [ (e^{2 x}-1)f_1^2 +  (e^{2 y}-1)f_2^2  +  (e^{2 z}-1)f_3^2 ]} -1$, \\
$G_2(u,f)$ &$= G_1(u,\tau_1(f)) + G_1(u,\tau_2(f)) + G_1(u,\tau_3(f))$.
\end{tabular}
 \end{center}

 We then use $G(u,f)$ to define
 
 \begin{center}
 \begin{tabular}{ll}

 $T_1(u)$ &$=\sum_{f \in \mu_F} G(u, f)= (\#\mu_F)\cdot G(u, 1),$ \\
 $T_2(u)$ &$=\sum_{f \in O_F, \hspace*{0.1cm}\|f\|^2 \ge 22 } G(u, f),$ \\ $T_3(u)$&$=\sum_{f \in O_F\backslash\mu_F, \hspace*{0.1cm}\|f\|^2 < 22 } G(u, f)$.
 \end{tabular}
 \end{center}

\begin{proposition}\label{equiv} Theorem \ref{thmmain} holds if and only if
$T_1(u) + T_2(u) + T_3(u) <0 \text{ for all } u= (e^x, e^y, e^z) \ne (1,1,1).$
\end{proposition}

\begin{proof}
Since $3[k^0(D)-k^0(D_0)]/\|w\|^2=\sum_{f \in O_F } G(u, f)  = T_1(u) + T_2(u) + T_3(u)$ by  \cite[Proposition 4.1]{TranPeng1}, the result follows. 
\end{proof}

We now establish Theorem \ref{thmmain} in this case by proving several results which are achieved using the Galois property of $F$, the Taylor expansion of $e^t$ and the symmetry of  $G_2(u,f)$.

\begin{lemma}\label{G} For all $u \in \mathbb{R}^3_+$ and $ f\in O_F$,
	\[G(u,\tau_1(f)) = G(u,\tau_2(f)) =  G(u,\tau_3(f)).\] 
\end{lemma}
\begin{proof}
	This can be seen from the formulas of $G(u,\tau_i(f))$, $i\in\{1, 2, 3\}$ and the fact that $\|\tau_1(f) \| = \| \tau_2(f)\| = \|\tau_3(f)\|$ for all $ f\in O_F$.
\end{proof}

\begin{proposition}\label{taylor1}
	Let $ \|w\|^2 = 2(x^2 + y^2 +z^2)>0$. Then for all  $f \in O_F$,
	$$G(u, f) \le 4 \pi^2 \|f^2\|^2 e^{-\pi \|f\|^2} \left(1 + \frac{1}{2} e^{2 \pi \|w\| \|f ^2\|} \right).$$
	In particular,  if $f \in O_F$ with $\|f\|^2 \ge 22$ then
	$$G(u, f) \le 2 \pi^2\left(e^{-(\pi-2/7)\|f\|^2}+ \frac{1}{2} e^{-(\pi-2\pi\|w\|-2/7) \|f\|^2} \right).$$
\end{proposition}

\begin{proof}
	The first inequality is from  \cite[Proposition 4.2]{TranPeng1}. The second is obtained from the first combined with
	$$\|f^2\|^2 \le \frac{1}{2}\|f\|^4 \le \frac{1}{2} e^{2\|f\|^2/7} \text{ and }  \|f^2\| \le \|f\|^2 \text{ for all } \|f\|^2 \ge 22.$$
	\vspace*{-1em}
\end{proof}

\begin{proposition}\label{Gat1}
	If $\|w\|^2 \in (0, 0.24163^2)$ then $T_1(u) < -98.4664 \cdot 10^{-9} \cdot(\# \mu_F)$.
\end{proposition}

\begin{proof}
	As $T_1(u) =(\# \mu_F)\cdot G(u, 1)$, it is sufficient to prove $G(u, 1)< -98.4664 \cdot 10^{-9}.$ 
	Since $ 0< \|w\|<0.24163 $, 
	$$2(e^{2 x} + e^{2 y} +  e^{2 z}-3 )\ge 1.9 \cdot 2 (x^2 + y^2 + z^2)=1.9 \|w\|^2.$$
	Consequently,
	$$G_1(u,1)= e^{-2\pi [ e^{2 x} + e^{2 y} +  e^{2 z}-3]} -1 \le e^{- 1.9 \pi \|w\|^2} -1.$$ The bounds on  $\|w\|$ also imply
	$$ G_2(u,1)/\|w\|^2 = 3   G_1(u,1)/\|w\|^2 \le 3   [e^{- 1.9 \pi \|w\|^2} -1]/\|w\|^2 < -15.1198.$$
	Therefore  
	$G(u, 1)= e^{-6 \pi} G_2(u,1)/\|w\|^2 < -98.4664 \cdot 10^{-9}.$
\end{proof}

\begin{proposition}\label{sum2}
	If $\|w\|^2  \in (0,0.24163^2)$ then $T_2(u)< 2.19278\cdot 10^{-9}$.
\end{proposition}

\begin{proof}
	By Proposition \ref{taylor1}, one has
	$$T_2(u) \le 2 \pi^2 \sum_{f \in O_F, \hspace*{0.1cm}\|f\|^2 \ge 22 } e^{-(\pi-2/7)\|f\|^2}+  \pi^2 \sum_{f \in O_F, \hspace*{0.1cm}\|f\|^2 \ge 22 }  e^{-(\pi-2\pi\|w\|-2/7) \|f\|^2} .$$
	The first sum is at most $10^{-23}$ by Corollary \ref{sum}. 
	Further, since $\|w\|<0.24163$,
	$$\pi-2\pi\|w\|-2/7 \ge \pi-2 \cdot 0.24163\hspace*{0.1cm}\pi-2/7.$$
	The second sum is then bounded by $2.19277\cdot 10^{-9}$ by Corollary \ref{sum}. Thus $T_2(u) \le 2.19278\cdot 10^{-9}$.
\end{proof}

Bounding $T_3(u)$ is more technical than bounding $T_1(u)$ and $T_2(u)$ as accomplished above. We bound $T_3(u)$ in the following proposition, breaking the proof into several cases.

\begin{proposition}\label{sum3}
	If $\|w\|^2 \in (0, 0.24163^2)$ then  \[T_3(u) <98.4664 \cdot 10^{-9} \cdot(\# \mu_F) -2.19278\cdot 10^{-9}.\]
\end{proposition}

\begin{proof}
	Since $\# \mu_F \ge 2$,  we have 
	$$98.4664 \cdot 10^{-9} \cdot(\# \mu_F) -2.19278\cdot 10^{-9} > 1.9474 \cdot 10^{-7}.$$ Therefore it is sufficient prove $T_3(u)< 1.9474 \cdot 10^{-7}$.
	
	For $f \in O_F$ define the lengths $l_1$ and $l_2$ by $l_1= \|f\|^2$ and $l_2=\|f^2\|^2$.
	For $\|w\|\in (0, 0.24163)$ apply Proposition \ref{taylor1} to bound $ G(u, f)$ as a function of $l_1$ and $l_2$:
	\begin{equation}\label{esT3}
		G(u, f)\leq\mathcal{G}(l_1, l_2) = 4 \pi^2 l_2 e^{-\pi l_1} \left(  1 +   \frac{1}{2} e^{2  \pi \cdot 0.24163 \sqrt{l_2}} \right).
	\end{equation}
	
	Based on the results of Section \ref{secbound}, we divide our proof into 4 cases. \\
	
	\vspace*{0.1cm}
	\noindent\textbf{Case (1):} $p>61$ and either 
	\begin{itemize}
		\item $d \equiv 1, 2 \mod 4$ and $d>22$, or
		\item $d \equiv 3 \mod 4$ and $d>59$. 
	\end{itemize} 
	Using the result in Section \ref{secbound}, one can  show that $T_3(u) =0$ for these values.\\
	
	\vspace*{0.1cm}	
	\noindent\textbf{Case (2):} $p>61$ and either
	\begin{itemize}\item $d \equiv 1, 2 \mod 4$ and $d \le 22$, or	
		\item $d \equiv 3 \mod 4$ and $d \le 59$.
	\end{itemize}
	If $d>11$ then $T_3(u)=0$. We then can assume $d \le 11$. It leads to  $d\in\{1,2,3,7,11\}$. 
	
	If $d \in \{1, 2\}$, then any $f \in O_k$ has the form $ m + n \sqrt{-d}$, and  $\|f\|^2 =6(m^2 + n^2 d)$.
	
	If $d=1$, then $ f \in \{ \pm 1 \pm i\}, f^2 \in \{ \pm 2 i\}$, and $\| f\|^2 =12 $, $\|f^2\|^2 =24$. As a result, one has 
	$T_3(u)\le 4 \cdot \mathcal{G}(12, 24) \approx 1.4 \cdot 10^{-10}$.

	Similarly, if $d=2$, then either
	$ f = \pm 1 \pm \sqrt{-2}$ or $ f = \pm \sqrt{-2}$. When $ f = \pm 1 \pm \sqrt{-2}$ we have $f^2 = -1 \pm 2 \sqrt{-2}$,  $\|f\|^2 = 18$ and  $\|f^2\|^2 = 54$. When $ f = \pm \sqrt{-2}$, then $f^2 =2$. Thus $\| f\|^2 =12 $ and $\|f^2\|^2 =24$. As a consequence,   $T_3(u)\le 4 \cdot \mathcal{G}(18, 54)  + 2 \cdot \mathcal{G}(12, 24) < 10^{-10}$.
	
	For $d \in \{3, 7, 11\}$, we do the same computation:
	
	If $d=3$, then $f \in \{\pm 3/2 \pm \sqrt{-3}/2, \pm \sqrt{-3}\}$ and hence  $T_3(u)\le 6 \cdot \mathcal{G}(18, 54)   < 10^{-10}$.
	
	If $d=7$, then $f \in \{\pm 1/2 \pm \sqrt{-7}/2\}$ and thus $T_3(u)\le 4\cdot \mathcal{G}(12, 24)   < 1.4 \cdot 10^{-10}$.
	
	If $d=11$, then $f \in \{\pm 1/2 \pm \sqrt{-11}/2\}$ and hence $T_3(u)\le 4\cdot \mathcal{G}(18, 54)   < 10^{-10}$.\\

	\vspace*{0.1cm}
	\noindent\textbf{Case (3):} $p \le 61$ and either \begin{itemize}
		\item $d>22$ with $d \equiv 1, 2 \mod 4$ or
		\item $d>59$ with $d \equiv 3 \mod 4 $.
	\end{itemize}
	By Section \ref{secbound}, one implies $T_3(u) =0$ for $p>13$. Therefore, we only consider $p \in \{7, 9, 13\}$.  
	For each of these values for $p$ there is exactly one cyclic cubic field $K$ with conductor $p$. We can find all vectors $f \in O_K$ for which  $\|f\|_K^2 < 11$, equivalently, $\|f\|^2 < 22$ using an LLL-reduced basis of the lattice $O_K$ \cite[Section 12 ]{ref:1} or by applying the Fincke--Pohst algorithm  \cite[Algorithm 2.12]{ref:40}). 
	
	We first consider $p=7$. There are such 18 elements $f \in O_K$ such that $\|f\|^2 < 22$ as follows.
	\[
	\begin{array}{|c| c |c|}
			\hline
		\hspace*{0.5cm}\|f\|^2 \hspace*{0.5cm} & \hspace*{0.5cm}\|f^2\|  \hspace*{0.5cm} &  \text{Number of elements } f \\
		\hline
		10 & 26 & 6\\
		\hline
		12 & 52 & 6\\
		\hline
		20 & 132 & 6 \\	\hline
	\end{array}
	\]
	Therefore
	$$T_3(u)\le 6\cdot \mathcal{G}(10, 26) +  6\cdot \mathcal{G}(12, 52)  + 6\cdot \mathcal{G}(20, 132)  < 1.76 \cdot 10^{-7}.$$ 
	
	For $p=9$, we do a similar computation and obtain, 
	$$T_3(u)\le 6\cdot \mathcal{G}(12, 36) +  6\cdot \mathcal{G}(18, 98)  + 6\cdot \mathcal{G}(18, 138)  < 1.64 \cdot 10^{-9}.$$ 
	
	Finally when $p=13$, one has 
	$T_3(u)\le 6\cdot \mathcal{G}(18, 106) +  6\cdot \mathcal{G}(84, 120)   < 3 \cdot 10^{-14}.$\\
	
	\vspace*{0.1cm}
	\noindent\textbf{Case (4):} $p \le 61$ and either \begin{itemize} \item $d \equiv 1, 2 \mod 4$ and $d \le 22$, or	
		\item $d \equiv 3 \mod 4$ and $d \le 59$. 
	\end{itemize}
	
	Let 
	$\mathcal{L}_1 = \{ (7, 1), (9,1), (13,1), (19,1), (7,2), (9,2), (9,6), (13, 13), (7,14), (7, 21), (9,21)\},$
	$\mathcal{L}_2=\{ (7, 3), (9,3), (13,3), (19,3), (31, 3), (37,3), (43,3), (7,7), (9,7), (13,7), (19,7), (31, 7), $\\$	(7, 11), (9,11),  (13, 11), (9,15), (19, 19), (31, 31), (7, 35), (9, 39), (13, 39), $
	$ (43, 43), (9, 51)\}$. 
	By Proposition \ref{limpd1}, if $(p,d)$ is not in $\mathcal{L}_1 \cup \mathcal{L}_2$, then
	\[T_3(u) \le \sum_{f \in O_k\backslash\mu_F, \hspace*{0.1cm}\|f\|^2 < 22 } G(u, f) + \sum_{f \in O_K\backslash\mu_F, \hspace*{0.1cm}\|f\|^2 < 22 } G(u, f).\]
	The first sum is nonzero when $d \in \{1, 2, 3, 7, 11\}$ by Proposition \ref{bound} and the second is nonzero when $p\in \{7, 9, 13\}$ by Proposition \ref{bounp}. These sums are at most $1.4 \cdot 10^{-10}$ (see Case (2)) and $1.76 \cdot 10^{-7}$ (see Case (3)), respectively. Thus $T_3(u)< 1.9474 \cdot 10^{-7}$.
	
	We now consider the cases in which  $(p, d)\in\mathcal{L}_1 \cup \mathcal{L}_2$. Using an LLL-reduced basis of the lattice $O_F$ viewed as a lattice in $\mathbb{R}^6$ \cite[Section 12]{ref:1} or by applying the Fincke--Pohst algorithm  \cite[Algorithm 2.12]{ref:40}), we first list all elements $f \in O_F$ such that $\|f\|^2 < 22$. After that we compute the  the function $\mathcal{G}(l_1,l_2)$ to find an upper bound for $T_3(u)$ using Proposition \ref{taylor1}. as done in Case (2). The results are shown in Table \ref{table1}. 
	
	When $(p,d) \in \{(7,7), (7, 3), (7, 1)\}$, the number of roots of unity of $F$ is $\# \mu_F \in \{14, 6, 4\}$ respectively. Therefore, in these cases we still have that $T_3(u)<98.4664 \cdot 10^{-9} \cdot(\# \mu_F) -2.19278\cdot 10^{-9}$ as desired. For other values of $(p,d)$ in Table \ref{table1}, $T_3(u)< 1.9474 \cdot 10^{-7}$.
\end{proof}

{\tiny{
		\begin{landscape}
			\begin{table}
				\centering
				\captionof{table}{Computing an upper bound for $T_3(u)$ for $(p, d) \in \mathcal{L}_1 \cup \mathcal{L}_2$}
				\label{table1}
				\begin{tabular}{|c|c|c|c|c| |c|c|c|c|c|}
					\hline
					$(p, d)$ & $\hspace*{0.1cm}\|f\|^2 \hspace*{0.1cm}$ & $\hspace*{0.1cm}\|f^2\|  \hspace*{0.1cm}$ &   $\#$ elements $f $ & $T_3(u) \le $ & $(p, d)$ & $\hspace*{0.1cm}\|f\|^2 \hspace*{0.1cm}$ & $\hspace*{0.1cm}\|f^2\|  \hspace*{0.1cm}$ &   $\#$ elements $f $ & $T_3(u) \le $ \\
					\hline 
					$(7, 1)$  & 10  & 26 &  12   & $3.5200 \cdot 10^{-7}$   &
					$(7, 3)$  & 10  & 26  &  18  & $5.2784 \cdot 10^{-7}$  \\
					& 12  & 24 & 4   &  &  & 12  & 52  & 18     &   \\ 
					& 12  & 52 & 12  &   & & 14  & 42  & 36    &  \\    
					& 16  & 52 & 24   &  & & 16  & 52  & 36    &  \\    
					& 18  & 82 & 24   & &  & 18  & 54  & 6    &  \\    
					& 20  & 76 & 24   & &  & 18  & 82  & 36    &  \\    	
					& 20  & 104 & 12  &  &  & 20  & 132 & 18   & \\    	
					& 20  & 132 & 12   & &  & & & &  \\ 
					\hline 
					$(9,1)$   & 12  & 24 & 4   & 			$3.4064 \cdot 10^{-9}$ & 
					$(9,3) $  & 12  & 36 & 108   & 		$2.9425 \cdot 10^{-8}$ \\     		
					& 12  & 36 & 12  & & & 18  & 54 & 18 &  \\  
					& 18  & 66 & 24  & & & 18  & 66 & 108  &  \\
					& 18  & 90 & 12   & & & 18  & 90 & 54  &  \\
					& 18  & 138 & 12  & & & 18  & 138 & 54 &  \\      	
					\hline 
					$(13,1)$   & 12  & 24  & 4   &  			$1.3672 \cdot 10^{-9}$ &
					$(13,3) $   & 18  & 54  & 6   &  			$6. 4034 \cdot 10^{-14}$ \\    
					& 18  & 106  & 12 &  & & 18  & 106  & 18  & \\  
					& 20  & 84 & 12   & & & 20  & 84 & 18   &  \\ 
					\hline 
					$(19,1)$  & 12  & 24  & 4   & $1.3668 \cdot 10^{-10}$ &    
					$(19,3) $,  $(31,3), $ & 18  & 54  & 6   & 	$1.2367 \cdot 10^{-16}$ \\
					& & & &  &  $(37,3) $ or  $(43,3) $ & & &  &   \\      
					\hline 
					$(7,2)$  & 10  & 26 & 6  & $1.7600 \cdot 10^{-7}$ & $(7,7) $  & 10  & 26 & 42   & $1.2326 \cdot 10^{-6}$ \\   
					& 12  & 24 & 2   &  & & 12  & 24 & 28   &  \\ 
					& 12  & 52 & 6   &  & & 12  & 52 & 42   &  \\   
					& 18  & 54 & 4   & & & 14  & 42 & 42   &  \\   
					& 20  & 104 & 6   & & & 20  & 76 & 84   &  \\ 	
					& 20  & 132 & 6   & & & 20  & 104 & 84   &  \\
					&     &     &     & & & 20  & 132 & 42  &  \\ 
					\hline        
					$(9,2)$  & 12  & 24  & 2   & 			$1.7032 \cdot 10^{-9}$ &
					$(9,7) $  & 12  & 24  & 4   & 			$1.7716 \cdot 10^{-9}$ \\ 
					& 12  & 36  & 6   & & & 12  & 36  & 6   &  \\
					& 18  & 54  & 4   & & & 18  & 90  & 6   &  \\
					& 18  & 90  & 6   &  & & 18  & 138 & 6   &  \\  
					& 18  & 138 & 6   &  & &     &     &  &    \\
					\hline 
					$(9,6)$     & 12  & 36  & 6   &  			$1.6349 \cdot 10^{-9}$ &
					$(13,7) $  & 12  & 24  & 4   & 			$1.3670 \cdot 10^{-10}$ \\        		  
					and	     & 18  & 90  & 6   & & & 18  & 106 & 6   &  \\  
					$(9,21)$ & 18  & 138 & 6   & & & 20  & 84  & 6   &  \\   
					\hline 
					$(13, 13) $   & 18  & 106  & 6   &  			$2.1304 \cdot 10^{-14}$ &
					$(7,14)$ 	  & 10  & 26 	& 6   &  			$1.7593 \cdot 10^{-7}$ \\
					& 20  & 84  & 6 &  & 
					and 	& 12  & 52 	& 6   &  \\ 
					& &     &  & &   $(7,21)$	& 20  & 132 & 6   &  \\    	 
					\hline             
				\end{tabular}
			\end{table}
		\end{landscape}
}}




\section{A comparison to previous work}\label{sec5}
 Compared to previous papers \cite{ref:14,ref:15, Tran2,TranPeng1}, many of the details of our proofs are more technical and required us to develop new techniques, though the overall structure of our proof is similar to previous work. In terms of structural similarity, we considered the case where $I$ is principal separately from the case where it is not. When $I$ is  principal, we subdivide further based on the length of $w =-\log(u)$.

The techniques that are unique to this paper are:
\begin{itemize}
	\item We split $h^0(D)$ into three summations instead of two as in previous papers. The reason is that the upper bound for 
	\[\sum_{\substack{ f \in I \\ \| u f\|^2 > 6\cdot 2^{1/3} }}e^{-\pi \|u f\|^2 }\]
	obtained by Lemma \ref{err} is too large to be useful for imaginary cyclic sextic fields. To solve this problem, we split this sum into $\Sigma_2 + \Sigma_3$ (see \ref{eqsum}) and find an efficient bound for $\Sigma_2$. We compute this bound in Section \ref{case1} and the proof of Lemma \ref{S1}. The bound is a multiple of $\#\mu_F$. To bound $\Sigma_2$, we find an upper bound for  $\#S_{2}$, which is the primary goal of Section \ref{counting1} (see Propositions \ref{boundm2}, \ref{boundm1}, \ref{boundm1p=7} ) which are not given in previous work. Further, when $I$ is principal, $\Sigma_1$ can be bounded by a multiple of $\#\mu_F$, minus a constant (see Lemma \ref{S1}).

	\item  In this paper $T_1(u)$ (Proposition \ref{Gat1}) and $T_3(u)$ (Proposition \ref{sum3}), depend on a multiple of $\#\mu_F$, in contrast to the cyclic cubic case \cite{TranPeng1}. Using $\#\mu_F$ is necessary in order to show that $T_1(u) + T_2(u) + T_3(u) <0$.

	\item To compute an efficient upper bound for $T_3(u)$, we bound the discriminant $\Delta_F$ of $F$. Previous papers have used two different approaches to achieve this goal: \\

\noindent\textbf{(1) The complex quartic case:} This case  uses a short fundamental unit $\varepsilon$ to bound $\Delta_F$, where $\|\varepsilon\|^2<1+\sqrt{2}$ \cite[Lemma 6.3]{Tran2}. In the imaginary cyclic sextic case, we cannot take this approach because the unit group of $F$ depends on $p$ but not $d$ (see Lemma \ref{unitF}) and hence does not depend on $\Delta_F$. Therefore we cannot bound for $\Delta_F$ based on the size of the units of $F$. When $p$ is fixed the fundamental units are fixed, yet we can make $\Delta_F$ as large as we want by choosing a large value for $d$ (Lemma \ref{disF}). \\
	
\noindent\textbf{(2) The cyclic cubic case:} This case uses the existence of short elements $f \in O_F\backslash \mu_F$ with $\|f\|^2<10$ \cite[Proposition 2.3]{TranPeng1}. For imaginary cyclic sextic fields there was no existing result to bound the length of an element in $O_F$ in terms of $\Delta_F$, so we develop this in Section \ref{secbound}. We show that if $F$ has short elements $f \in O_F\backslash \mu_F$, where short means that $\|f\|^2<22$, then one of three things holds: (i) $p \le 61$ and $d \le 59$, (ii) $f$ is in the quadratic subfield $k= \mathbf{Q}(\sqrt{-d})$ with $d \le 11$, or (iii) $f$ is in the cubic subfield $K$ with conductor $p \le 13$. 
	
	Using upper bounds for $p$ and $d$, we can list all imaginary cyclic sextic fields in which there exist elements $f\in O_F\backslash \mu_F$ with $\|f\|^2<22$. For each such field, we compute these short elements and make use of the function $\mathcal{G}$ in (\ref{esT3}) to find an upper bound for $T_3(u)$ (see Table \ref{table1}) .

\end{itemize}


\section*{Acknowledgement}
The authors would like to thank  Ren\'{e} Schoof for a useful discussion. The authors also would like to thank the reviewer for their insightful comments that helped improve the manuscript.  
Ha T. N. Tran acknowledges the support of the Natural Sciences and Engineering Research Council of Canada (NSERC) (funding RGPIN-2019-04209 and DGECR-2019-00428). 
Peng Tian is supported by National Natural Science Foundation of China (grant $\#$11601153) and  Fundamental Research Funds for the Central Universities (grant $\#$222201514319).



\bibliographystyle{abbrv}   
\bibliography{myrefs}

\end{document}